\documentclass[11pt]{article}

\usepackage{latexsym,amsfonts,amsmath,amssymb,amsthm,url}
\usepackage[english]{babel}
\usepackage[usenames]{color}

\usepackage{chngcntr}
\counterwithin{equation}{section}
\newtheorem{itlemma}{Lemma}[section]
\newtheorem{itproposition}[itlemma]{Proposition}
\newtheorem{theorem}[itlemma]{Theorem}
\newtheorem{itcorollary}[itlemma]{Corollary}
\newtheorem{itremark}[itlemma]{Remark}
\newtheorem{itremarks}[itlemma]{Remarks}
\newtheorem{itdefinition}[itlemma]{Definition}
\newtheorem{itexample}[itlemma]{Example}

\newenvironment{fact}{\begin{itfact}\rm}{\end{itfact}}

\newenvironment{lemma}{\begin{itlemma}}{\end{itlemma}}
\newenvironment{remark}{\begin{itremark}\rm}{\end{itremark}}
\newenvironment{remarks}{\begin{itremarks} \rm}{\end{itremarks}}
\newenvironment{corollary}{\begin{itcorollary}}{\end{itcorollary}}
\newenvironment{proposition}{\begin{itproposition}}{\end{itproposition}}
\newenvironment{definition}{\begin{itdefinition}\rm}{\end{itdefinition}}
\newenvironment{example}{\begin{itexample}\rm}{\end{itexample}}
\newcommand{\bl}[1]{\begin{lemma}\label{#1}}
\newcommand{\br}[1]{\begin{remark}\label{#1}}
\newcommand{\brs}[1]{\begin{remarks}\label{#1}}
\newcommand{\bt}[1]{\begin{theorem}\label{#1}}
\newcommand{\bd}[1]{\begin{definition}\label{#1}}
\newcommand{\bp}[1]{\begin{proposition}\label{#1}}
\newcommand{\bc}[1]{\begin{corollary}\label{#1}}
\newcommand{\bfact}[1]{\begin{fact}\label{#1}}
\newcommand{\bex}[1]{\begin{example}\label{#1}}
\newcommand{\ec}{\end{corollary}}
\newcommand{\efact}{\end{fact}}
\newcommand{\eex}{\end{example}}
\newcommand{\el}{\end{lemma}}
\newcommand{\er}{\end{remark}}
\newcommand{\ers}{\end{remarks}}
\newcommand{\et}{\end{theorem}}
\newcommand{\ed}{\end{definition}}
\newcommand{\ep}{\end{proposition}}
\newcommand{\E}{\mathbb{E}}
\newcommand{\N}{\mathbb{N}}

\newcommand{\Z}{\mathbb{Z}}

\newcommand{\pp}{\mathbb{P}}

\newcommand{\kB}{\mathcal{B}}

\newcommand{\kR}{\mathcal{R}}
\def\R{{\mathcal R}}
\newcommand{\kO}{\mathcal{O}}

\newcommand{\kF}{\mathcal{F}}
\newcommand{\kE}{\mathcal{E}}
\newcommand{\kH}{\mathcal{H}}

\newcommand{\kN}{\mathcal{N}}

\newcommand{\lin}{\left[\kern-0.15em\left[}
\newcommand{\rin} {\right]\kern-0.15em\right]}
\newcommand{\ilin}{\left]\kern-0.15em\left]}
\newcommand{\irin} {\right[\kern-0.15em\right[}

\def\reff#1{(\ref{#1})}
\def \ind {\hbox{1\hskip -3pt I}}

\def\var{\text{Var}}
\def\bs{\backslash}

\newtheorem {theo} {Theorem} [section]

\newtheorem {lem} [theo] {Lemma}

\newtheorem {rem} [theo] {Remark}

\begin{document}

%\title{{\bf Boundary of the Range I: Typical Behavior.}}
\title{{\bf Boundary of the Range of Transient Random Walk}}
\author{
\normalsize{\textsc{Amine Asselah}\footnote{LAMA, UPEC \& IM\'eRA
E-mail: amine.asselah@u-pec.fr}\ \ \ \&\ \ \textsc{Bruno
Schapira}\footnote{Aix-Marseille Universit\'e, CNRS, Centrale Marseille, I2M, UMR 7373, 13453 Marseille, France. E-mail: bruno.schapira@univ-amu.fr}}}

\date{}
\maketitle

\begin{abstract} We study the boundary of the range of
simple random walk on $\Z^d$ in the transient case $d\ge 3$. 
We show that volumes of the range and its boundary
differ mainly by a martingale. As a consequence, we
obtain an upper bound on the variance of order $n\log n$ in dimension
three. We also establish a Central Limit Theorem in dimension four and
larger.
\end{abstract}

\section{Introduction}
Let $(S_n,\ n\ge 0)$ be a simple random walk on $\Z^d$.
Its range $\R_n=\{S_0,\dots,S_n\}$ is a familiar object of Probability
Theory since Dvoretzky and Erd\"os' influential paper \cite{DE}. 
The object of interest in this paper is the boundary of the range
\begin{equation}\label{def-bound}
\partial \R_n=\{ x\in \R_n:\ \text{there exists}\ y\sim x\
\text{with }y\not\in \R_n\},
\end{equation}
where $x\sim y$ means that $x$ and $y$ are at (graph) distance one.
Our interest was triggered by a recent paper of Berestycki and
Yadin \cite{BY} which proposes a model of
hydrophobic polymer in an aqueous solvent, consisting of tilting
the law of a simple random walk by
$\exp(-\beta |\partial \R_n|)$. One interprets the range as the space
occupied by the polymer, and its complement as the space occupied
by the solvent. Hydrophobic means that the monomers dislike the
solvent, and the polymer tries to minimize the {\it boundary of the range}.
The Gibbs' weight tends to minimize contacts
between the monomers and the solvent, and the steric effect has 
been forgotten to make the model mathematically tractable. 
Besides its physical appeal,
the model gives a central role to the boundary of the range,
an object which remained mainly in the shadow until recently. 
To our knowledge it first appeared in the study of the entropy of the range of a simple random walk
\cite{BKYY}, with the conclusion that in dimension 
two or larger, the entropy of the range scales 
like the size of the boundary of the range.
Recently, Okada \cite{Ok1} has established a law of large numbers 
for the boundary of the range for a transient
random walk, and has obtained bounds on its expectation in dimension two.
\bt{theo-okada}[Okada] 
Consider a simple random walk in dimension $d=2$. Then
\begin{equation}\label{okada-1}
\frac{\pi^2}{2}\le
\lim_{n\to\infty} \frac{\E\big[|\partial \R_n|\big]}{n/\log^2(n)}\le 2\pi^2,
\end{equation}
where part of the result is that the limit exists. 
Moreover, when $d\ge 3$, almost surely
\begin{equation}\label{okada-2}
\lim_{n\to\infty}  \frac{|\partial \R_n|}{n}=
\pp\big(\{z:\ z\sim 0\}\not\subset \R_\infty\cup \widetilde \R_\infty,\ 
H_0=\infty\big),
\end{equation}
where $\R_\infty$ is the range of a random walk in an infinite
time horizon, and $H_0$ is the hitting time of 0, whereas
quantities with tilde correspond to those of an independent copy
of the random walk.
\et
\vspace{0.2cm}
The range of a random walk has nice properties: (i) it is an
increasing function of time, (ii) the event that $S_k$ belongs to
$\R_n$ for $k\le n$ is $\sigma(S_0,\dots,S_k)$-measurable, 
(iii) the volume of the range $\R_n$ is the union of
the collection of sub-ranges $\{S_k,k\in I\}$ as
$I$ runs over a partition of $[0,n]$.
A little thought shows that the boundary of the range shares none of
these properties, making its study more difficult.
The thrust of our study is to show that for a transient random
walk, range and boundary of the range are nonetheless 
correlated objects.
Indeed, we present two ways to
appreciate their similar nature. 
On one hand the sizes of the boundary of the range and some range-like sets defined below (the $\R_{n,V}$) differ 
mainly by a martingale. 
On the other hand, we show that the boundary of the range, as the range itself, can be analyzed through a dyadic decomposition of the path. To make the first statement precise,
we need more notation. Let $V_0=\{z:\ z\sim 0\}$, be the neighbors
of the origin, and for any nonempty
subset $V$ of $V_0$, let $\R_{n,V}$ be the set of sites of $\Z^d$
whose first visit occurs at some time $k\le n$, and such that
$(S_k+V_0)\cap \R^c_{k-1}=S_k+V$. In particular, 
$\R_{n,V}$ behaves like the range in the sense that properties
(i)-(ii) listed above do hold,
and as we will see below, their variance can be bounded using the 
same kind of techniques as for the range.

Note also that $\R_n$ is the disjoint union of the $\R_{n,V}$, 
with $V$ subset of $V_0$.   
We are now ready for our first observation.
\bp{prop-martin} There is a martingale $(M_n,\ n\in \N)$, adapted
to the standard filtration such that for any positive integer $n$,    
\begin{equation}\label{equality}
|\partial \R_n|=
\sum_{V\subset V_0} \rho_V\, |\R_{n-1,V}|
+ M_n+  \kE_{n},
\end{equation}
with $\rho_\emptyset=0$ and for any non-empty $V$ in $V_0$
\begin{equation}\label{def-rho}
\rho_V=\pp\big(V\not\subset \R_{\infty}\big)\quad\text{and}\quad
\E\left(  \kE_n^2\right)= \left\{ \begin{array}{ll}
\kO(n) & \textrm{if }d=3\\
\kO(\log^3 (n)) & \textrm{if }d=4\\
\kO(1) & \textrm{if }d\ge 5.
\end{array}
\right.
\end{equation}
\ep
\begin{rem}\label{rem-doob}
The decomposition \eqref{equality} is simply Doob's decomposition 
of the adapted process $|\partial \R_n| - \kE_n$, as we see 
more precisely in Section \ref{sec-martingale}. The key observation
however is that the increasing process (in Doob's decomposition) 
behaves like the range.
\end{rem}

Jain and Pruitt \cite{JP} have established a Central Limit
Theorem for the range in dimension three with a variance
scaling like $n\log n$. Proposition~\ref{prop-martin} 
makes us expect that the boundary of the range has a similar
behavior. Indeed, we establish the following estimate on
the mean square of the martingale. This estimate is delicate,
uses precise Green's function asymptotics, and the symmetry of the walk.
It is our main technical contribution.
\bp{propMn}
There are positive constants $\{C_d,d\ge 3\}$, such that
\begin{eqnarray*}
\var\, (M_n) \le \left\{ \begin{array}{ll} 
C_3\,  n \log n & \textrm{if }d=3\\
C_d\, n &  \textrm{if }d\ge 4. 
\end{array}
\right.
\end{eqnarray*}
\ep
Also, following the approach of Jain and Pruitt \cite{JP},
we establish the following estimate on the range-like object $\R_{n,V}$.
\bp{propRn}
Assume that $d=3$, and let $V$ be a nonempty subset of $V_0$. There is a positive constant $C$, such that 
\begin{equation}\label{estim-variance}
\var\, (|\R_{n,V}|)\le C\, n \log n.
\end{equation}
\ep
Then, a useful corollary of Propositions~\ref{prop-martin}, 
\ref{propMn} and \ref{propRn} is the corresponding bound for 
the variance of the boundary of the range in dimension $3$. 
\bt{theo-vard3}
Assume that $d=3$. Then, there is a positive constant $C$, such that 
\begin{equation}\label{main-vard3}
\var\, (|\partial \R_n|) \, \le C\, n\log n.
\end{equation}
\et
\begin{rem} \label{rem-Jain} 
Using the approach of Jain and Pruitt \cite{JP}, it is
not clear how to obtain a Central Limit Theorem for $\R_{n,V}$
(see Remark~\ref{final.rem} of the Appendix).
\end{rem}

Now the boundary of the range has a decomposition
similar to the classical Le Gall's decomposition \cite{LG} in terms of
intersection of independent ranges. This decomposition, though
simple, requires more notation to be presented. 
For integers $n,m$
let $\R(n,n+m)=\{S_k-S_n\}_{n\le k\le n+m}$, with the
shorthand notation $\R_n=\R(0,n)$, and note that
\[
\R(0,n+m)=\R(0,n)\cup\big(S_{n}+\R(n,n+m)\big).
\]
Observe that $\overleftarrow{\R}(0,n):=-S_n+\R(0,n)$ and 
$\R(n,n+m)$ are independent
and that by the symmetry of the walk $\overleftarrow{\R}(0,n)$ 
(resp. $\R(n,n+m)$) has the same law as $\R(0,n)$ (resp. $\R(0,m)$): it  
corresponds to the range of a walk seen backward from position $S_n$.
Finally, note the well known decomposition
\begin{equation}\label{range-equal}
|\R(0,n+m)|=|\R(0,n)|+|\R(n,n+m)|-|\overleftarrow{\R}(0,n)\cap\R(n,n+m)|.
\end{equation}
Equality \reff{range-equal} is the basis of Le Gall's celebrated
paper \cite{LG} on the range of recurrent random walk. 
It is also a key ingredient in most work
on self-intersection of random walks
(see the book of Chen \cite{C}, for many references).

To write a relation as useful as \reff{range-equal} for the
boundary of the range, we introduce more notation. For $\Lambda
\subset \Z^d$, we denote $\Lambda^+=\Lambda+ \overline V_0$, with $\overline V_0 = V_0 \cup \{0\}$, and we define its boundary as 
$$\partial \Lambda = \{z\in \Lambda\ : \ \exists y\in \Lambda^c \textrm{ with } y\sim z\}.$$
Now, our simple observation is as follows.
\bp{prop-legall}
For any integers $n,m$
\begin{equation}\label{boundary-ineq}
0\ge |\partial \R(0,n+m)|-\big(|\partial \R(0,n)|+|\partial \R(n,n+m)|\big)
\ge -Z(n,m),
\end{equation}
with
\begin{equation}\label{def-Z}
Z(n,m)=|\overleftarrow{\R}(0,n)\cap \R^+(n,n+m)|
+|\overleftarrow{\R}^+(0,n)\cap \R(n,n+m)|.
\end{equation}
\ep
We focus now on consequences of this simple decomposition.
For $d\ge 3$, we define functions $n\mapsto\psi_d(n)$, 
with the following dimension depending growth
\begin{equation}\label{def-psi}
\psi_3(n)= \sqrt{ n},\qquad \psi_4(n)= \log n,\qquad
\text{and for }d>4,\qquad \psi_d(n)=1.
\end{equation}
An essential step for a Central Limit Theorem, is to establish a 
linear lower bound on the variance.
Our bounds hold in dimension three and larger.
\bp{prop-vard3}
Assume that $d\ge 3$. There are positive constants $\{c_d,\ d\ge 3\}$, 
such that 
\begin{equation}\label{min-vard3}
\var\, (|\partial \R_n|) \, \ge c_d\, n.
\end{equation}
\ep
The idea behind the linear lower bound \reff{min-vard3} is to show
that there is a {\it clock process} whose fluctuations
are normal (on a scale square root of the time elapsed),
and which is independent of the boundary of the range process.
Thus, typical fluctuations of the clock process,
provoke a time change {\it at constant} boundary of the range.
Note that in dimension $3$, this technique does not allow to obtain a lower bound of order $n\log n$, matching our upper bound (see also  
Remark \ref{final.rem} for some additional comment on this). 

We now formulate our main Theorem.
\bt{prop-dyadic} When dimension is larger than or 
equal to three, there are constants $\{C_d,d\ge 3\}$, such that
for any positive integer $n$ 
\begin{equation}\label{prop-dyad-1}
\frac{C_d\psi_d(n)}{n}\ge \frac{\E[|\partial \R_n|]}{n}-
\lim_{k\to\infty} \frac{\E[|\partial \R_k|]}{k}\ge 0.
\end{equation}
Assume now that the dimension is four or larger. Then,
the limit of $\var( |\partial \R_n|)/n$ exists, is positive, and for all $n\ge 1$, 
\begin{equation}\label{prop-dyad-2}
\left|\frac{\var(|\partial \R_n|)}{n}-
\lim_{k\to\infty} \frac{\var(|\partial \R_k|)}{k}\right|\le 
\frac{C_d \sqrt n\psi_d(n)}{n}.
\end{equation}
Moreover, a standard Central Limit Theorem holds for $|\partial \kR_n|$.
\et
\br{rem-walk}
We have stated our results for the simple random walk,
but they hold, with similar proofs, for walks with
symmetric and finitely supported increments.
\er
Okada obtains also in \cite{Ok1} a large deviation principle 
for the upper tail 
(the probability that the boundary be larger than its mean), 
and in \cite{Ok2} he studies the most frequently visited 
sites of the boundary, 
and proves results analogous to what is known for the range. 

In a companion paper \cite{AS}, we obtain
large deviations for the lower tail, and provide
applications to phase transition for a properly normalized
Berestycki-Yadin's polymer model.

\vspace{0.2cm}
The rest of the paper is organized as follows. 
In Section~\ref{sec-notation}, we fix notation, 
recall known results on the Green's function, 
and prove a result about covering a finite subset.
In Section~\ref{sec-martingale}, we establish the Martingale decomposition
of Proposition~\ref{prop-martin} and prove
Proposition~\ref{propMn}.
We prove Proposition \ref{prop-vard3} in Section~\ref{sec-LBvar}. 
In Section~\ref{sec-legall},
we present the dyadic decomposition for the boundary 
of the range and deduce 
Theorem \ref{prop-dyadic}, using Le Gall's argument. 
Finally in the Appendix, we prove Proposition~\ref{propRn}.

\section{Notation and Prerequisites}\label{sec-notation}
For any $y,z\in \Z^d$, we denote by $\|z-y\|$ the Euclidean norm between $y$ and $z$, and by $\langle y,z\rangle$ 
the corresponding scalar product. 
Then for any $r>0$ we denote by $B(z,r)$ the ball of radius $r$ centered at $z$:
$$B(z,r):=\{y\in \Z^d\ :\ \|z-y\| \le r\}.$$ 
For $x\in \Z^d$, we let $\pp_x$ be the law of the random walk starting from $x$, and denote its standard filtration by $(\kF_k,k\ge 0)$. 
For $\Lambda$ a subset of $\Z^d$ we define the hitting time of $\Lambda$ as   
$$H_\Lambda:=\inf\{n\ge 1\ :\ S_n\in \Lambda\},$$
that we abbreviate in $H_x$ when $\Lambda$ is reduced to a single point $x$. 
Note that in this definition we use the convention to consider only times larger than or equal to one. 
At some point it will also be convenient to consider a shifted version, so we also define for $k\ge 0$, 
\begin{eqnarray}
\label{shifthit}
H_\Lambda^{(k)}:=\inf\{n\ge k\ :\ S_n\in \Lambda\}.
\end{eqnarray}
We will need bounds on the heat kernel, so let us recall a standard result: 
\begin{eqnarray}
\label{kernel}
\pp(S_n=z) \le C\, \frac 1{n^{d/2}}\, \exp(-c\|z\|^2/n)\quad \textrm{for all $z$ and $n\ge 1$}, 
\end{eqnarray}
for some positive constants $c$ and $C$ (see for instance \cite{HSC}). 
Now we recall also the definition and some 
basic properties of Green's function. 
For $u,v\in \Z^d$, the Green's function is
$$G(u,v)=\E_u\Big[\sum_{n\ge 0} \ind\{S_n=v\}\Big]=\pp_u[H_v<\infty]\cdot G(0,0),$$
and we use extensively the well-known bound (see \cite[Theorem 4.3.1]{LL}): 
\begin{eqnarray}
\label{Green}
G(0,z) =\kO\left(\frac1{1+\|z\|^{d-2}}\right).
\end{eqnarray}
We also consider Green's function restricted to a set $A\subset \Z^d$, 
which for $u,v\in A$ is defined by 
$$G_A(u,v)=\E_u\Big[\sum_{n=0}^{H_{A^c}-1} \ind\{S_n=v\}\Big].$$
We recall that $G_A$ is symmetric (see \cite[Lemma 4.6.1]{LL}): 
$$G_A(u,v)=G_A(v,u)\quad \textrm{for all }u,v\in A,$$
and that $G$ is also invariant by translation of the coordinates:   
$G(u,v)= G(0,v-u)$. 
Also, for $n\ge 0$, 
\begin{equation}
\label{def.Gn}
G_n(u,v)=\E_u\Big[\sum_{k=0}^n \ind\{S_n=v\}\Big].
\end{equation} 
It is well known (use \reff{Green} and Theorem 3.6 of \cite{LG}) 
that for $\psi_d$ defined in \reff{def-psi}, we have,
for some positive constants $\{C_d,d\ge 3\}$
\begin{equation}\label{bubble-G}
\sum_{z\in \Z^d} G^2_n(0,z)\le C_d\, \psi_d(n).
\end{equation}
We can now state the main result of this section.
\begin{lem}\label{lemLambda}
Let $\Lambda$ be a fixed finite subset of $\Z^d$, and fix $z\in \Lambda$. 
Then, there is a constant $c(\Lambda)$, such that for any two neighboring sites $y\sim y'$, 
\begin{eqnarray}
\label{eq-cover}
\pp_{y}(\Lambda\subset \kR_\infty)
-\pp_{y'}(\Lambda\subset \kR_\infty)=c(\Lambda)
\frac{\langle y'-y,y-z \rangle}{\|y-z\|^d}+\kO\left(\frac{1}{\|y-z\|^d}\right).
\end{eqnarray}
Moreover, 
\begin{eqnarray*}
\label{def-constant}
c(\Lambda)=\frac{1}{dv_d}\, 
\sum_{x\in \Lambda}\sum_{v\notin \Lambda}\, 
\ind_{\{v\sim x\}} 
\pp_{v}\big(H_\Lambda=\infty \big) \, 
\pp_x\big( \Lambda \subset \kR_\infty  \big),
\end{eqnarray*}
where $v_d$ denote the volume of the unit ball in $\mathbb R^d$.
\end{lem}
\begin{proof} 
First, since $\Lambda$ is finite, and \reff{eq-cover} is
an asymptotic result, we can always assume that $y$ and $y'$ 
do not belong to $\Lambda$. Now by a first entry decomposition
\begin{eqnarray}
\label{step-1}
\pp_{y}(\Lambda\subset \kR_\infty )=
\sum_{x\in \Lambda} \pp_y\big( S_{H_\Lambda}= x,\ H_\Lambda<\infty \big)
\pp_x\big(\Lambda \subset \kR_\infty \big).
\end{eqnarray}
Next, fix $x\in \Lambda$ and transform the harmonic measure
into the restricted Green's function (see for instance \cite[Lemma 6.3.6]{LL}): 
\begin{eqnarray*}
\label{cover-1}
\pp_y\big(S_{H_\Lambda}= x,\ H_\Lambda<\infty\big)=
\frac 1{2d}\, \sum_{v\in \Lambda^c,\, v\sim x} G_{\Lambda^c}(y,v)=\frac 1{2d}\, \sum_{v\in \Lambda^c,\, v\sim x} G_{\Lambda^c}(v,y). 
\end{eqnarray*}
Note also (see \cite[Proposition 4.6.2]{LL}) that
\begin{eqnarray*}
\label{cover-2}
G_{\Lambda^c}(v,y)=G(v,y)-\E_v\Big[\ind\{H_\Lambda<\infty \}\, 
G(S_{H_\Lambda},y)\Big]. 
\end{eqnarray*}
Therefore,
\begin{eqnarray}\label{cover-3}
\begin{split}
\pp_y\big(S_{H_\Lambda}= x, H_\Lambda<\infty \big)-&\pp_{y'}\big(S_{H_\Lambda}= x,H_\Lambda<\infty\big)=
\frac 1{2d}\, \sum_{v\in \Lambda^c,\, v\sim x}\big(G(v,y)-G(v,y')\big)\\
&-\frac 1{2d}\, \sum_{v\in \Lambda^c,\, v\sim x}  
\E_v\left[ \ind\{ H_\Lambda<\infty \}
\Big(G\big(S_{H_\Lambda},y\big)-G\big(S_{H_\Lambda},y'\big)\Big)\right].
\end{split}
\end{eqnarray}
Now, since $\Lambda$ is finite, we have (\cite[Corollary 4.3.3]{LL}) 
the expansion for any $z'\in \Lambda^+$ (recall that $z$ is a 
given site in $\Lambda$), 
\begin{eqnarray}
\label{cover-4}
G(z',y)-G(z',y')=\frac{2}{v_d} 
\frac{\langle y'-y,y-z\rangle}{\|y-z\|^d}
+\kO\left(\frac{1}{\|y-z\|^d}\right).
\end{eqnarray}
Combining \eqref{step-1}, \eqref{cover-3} and \eqref{cover-4}
we obtain the result \eqref{eq-cover}.
\end{proof}

\section{Martingale Decomposition}\label{sec-martingale}
In this Section, we establish Proposition~\ref{prop-martin}, as well
as Proposition~\ref{propMn} dealing with the variance of the martingale.
\subsection{Definition of the martingale and proof of Proposition~\ref{prop-martin}}
For $V$ nonempty subset of $V_0$ and $ k\ge 0$, let 
$$I_{k,V} = \ind\{S_k\notin \kR_{k-1}\textrm{ and } (S_k +V_0) \cap \kR_k^c = S_k+V\},$$
and
$$J_{k,V} = \ind\{(S_k +V) \nsubseteq  \{ S_j,\ j\ge k\}\}.$$ 
Then for $n\ge 1$, define 
$$J_{k,n,V} = \ind\{ (S_k +V) \nsubseteq  \{ S_k,\dots,S_n\}  \},$$
and 
\begin{eqnarray}
\label{bRnV}
\partial \kR_{n,V} = \{S_k\, :\, I_{k,V} J_{k,n,V} = 1,\ k\le n\}.
\end{eqnarray}
Note that $\partial \R_n$ is the disjoint union of the $\partial \R_{n,V}$, for $V$ non empty subset of $V_0$. 
Now instead of looking at $\sum_{k\le n} I_{k,V}J_{k,n,V}$ (which is equal to $|\partial \R_{n,V}|$), we look at 
$$Y_{n,V} = \sum_{k=0}^{n-1} I_{k,V}J_{k,V}.$$
However, since $Y_{n,V}$ is not adapted to $\kF_n$, we consider 
$$X_{n,V} = \E[Y_{n,V}\mid \kF_n],$$
and think of $X_{n,V}$ as a good approximation for $|\partial \R_{n,V}|$. So we define an error term as 
$$\kE_{n,V} := |\partial \kR_{n,V}| - X_{n,V}.$$
Now the Doob decomposition of the adapted process $X_{n,V}$ reads as 
$X_{n,V} = M_{n,V} + A_{n,V}$, 
with $M_{n,V}$ a martingale and $A_{n,V}$ a predictable process. Since 
$$X_{n,V} = \sum_{k=0}^{n-1} I_{k,V} \E[J_{k,V}\mid \kF_n],$$
we have 
$$A_{n,V} = \sum_{k=0}^{n-1} I_{k,V} \E[J_{k,V}\mid \kF_k].$$ 
Moreover, the Markov property also gives 
$$ \E[J_{k,V}\mid \kF_k] = \E[J_{k,V}] = \pp(V\nsubseteq \kR_\infty )= \rho_V,$$
for any $k\ge 0$. Therefore, 
\begin{eqnarray*}
|\partial \kR_{n,V}| = M_{n,V} + \rho_V |\kR_{n-1,V}| + \kE_{n,V}\qquad \textrm{for all }V\subset V_0,
\end{eqnarray*}
where we defined for $m\ge 0$, 
$$\kR_{m,V} = \{S_k\, :\, I_{k,V} = 1,\ k\le m\}.$$ 
Summing up $M_{n,V}$ over nonempty subsets of $V_0$ 
we obtain another martingale 
$M_n = \sum_{V\subset V_0} M_{n,V}$, 
and the error term $\kE_n=\sum_{V\subset V_0} \kE_{n,V}$, and  
since $|\partial \R_n|$ is also the sum over nonempty subsets of the $|\partial \R_{n,V}|$, 
we obtain the first part of Proposition~\ref{prop-martin}, namely Equation \eqref{equality}.

Now we prove \reff{def-rho}. 
First note that for any $k\le n-1$,  
$$\left|J_{k,n,V} - \E[J_{k,V}\mid \kF_n]\right| \le \pp_{S_n}(H_{S_k+V_0}<\infty)=\kO\left(\frac 1{1+\|S_n-S_k\|^{d-2}}\right),$$
using \eqref{Green} for the last equality. 
Then by using invariance of the walk by time inversion, we get  
\begin{eqnarray}\label{final-green}
\E[\kE_{n,V}^2] = \kO\left(\sum_{k,k'\le n} \, 
\E\left[\frac 1{(1+\|S_k\|^{d-2})(1+\|S_{k'}\|^{d-2})} \right]\right).
\end{eqnarray}
Moreover, by using the heat kernel bound \eqref{kernel}, we arrive at  
\begin{eqnarray}
\label{final-green2}
\E\left[\frac 1{(1+\|S_k\|^{d-2})^2}\right]= \left\{ \begin{array}{ll}
\kO(1/k) & \textrm{if }d=3\\
\kO((\log k)/k^2) &  \textrm{if }d=4\\
\kO(k^{-d/2}) &  \textrm{if }d\ge 5.
\end{array}
\right.
\end{eqnarray}
The desired result follows by using Cauchy-Schwarz. 

\subsection{Variance of the Martingale}
We establish here Proposition \ref{propMn}.
Let us notice that our proof works 
for $M_n$ only, and not for all the $M_{n,V}$'s.   
If we set for $n\ge 0$,
$$\Delta M_n = M_{n+1} - M_n,$$
then, Proposition \ref{propMn} is a direct consequence of the following result. 
\begin{eqnarray}
\label{deltaMn}
\E[(\Delta M_n)^2]= \left\{ \begin{array}{ll} 
\kO(\log n ) & \textrm{if }d=3\\
\kO(1) &  \textrm{if }d\ge 4. 
\end{array}
\right.
\end{eqnarray}
The proof of \eqref{deltaMn} is divided in three steps. 
The first step brings us to a decomposition of $\Delta M_n$ 
as a finite combination of simpler terms \eqref{Mn1}, 
plus a rest whose $L^2$-norm we show is negligible. In the second step, we observe that when we gather together 
some terms \eqref{Mn2}, their $L^2$-norm takes a particularly nice form \eqref{Mn3}. Finally in the third step we 
use these formula and work on it to get the right bound.

\vspace{0.2cm}
\noindent \underline{Step 1.} In this step, we just use the Markov property to write $\Delta M_n$ in a nicer way, 
up to some error term, which is bounded by a deterministic constant. 
Before that, we introduce some more notation. For $k\le n$, set  
$$I_{k,n,V} = \ind \left\{(S_k +V_0) \cap \{S_k,\dots,S_n\}^c = S_k+V\right\}.$$
The Markov property and the translation invariance of the walk show that for all $k\le n$
$$
\E[J_{k,V}\mid \kF_n] = \sum_{V'\subset V_0}
\ind_{\{V\cap V'\neq \emptyset\}} I_{k,n,V'}\,  \pp_{S_n-S_k}\left( V\cap V' \nsubseteq \kR_\infty\right).$$
Note that $I_{k,n,V'} \neq I_{k,n+1,V'}$ imples that $S_{n+1}$ and $S_k$
are neighbors. However, the number of indices $k$ such that
$S_{n+1}$ and $S_k$ are neighbors and $I_{k,V}=1$ is at most $2d$, since by definition of $I_{k,V}$ we only count the first visits to neighbors of $S_{n+1}$. 
Therefore the number of indices $k$ satisfying $I_{k,V} \neq 0$ and $I_{k,n,V'} \neq I_{k,n+1,V'}$, for some $V'$, is bounded by  
$2d$. As a consequence, by using also that terms 
in the sum defining $M_{n,V}$ are bounded 
in absolute value by $1$, we get
\begin{eqnarray*}
\Delta M_{n,V} &=& \sum_{k=0}^{n-1} I_{k,V}(\E[J_{k,V}\mid \kF_{n+1}] -\E[J_{k,V}\mid \kF_n]) +  I_{n,V}\E[J_{n,V}\mid \kF_{n+1}]\\
& =&   \sum_{k=0}^{n-1}  
\sum_{V'\subset V_0}\ind_{\{V\cap V'\neq \emptyset\}}
I_{k,V} \Big\{ I_{k,n+1,V'}\pp_{S_{n+1}-S_k} \left( V\cap V' \nsubseteq \kR_\infty\right)-  \\ &&  \qquad  I_{k,n,V'} \pp_{S_n-S_k}\left( V\cap V'\nsubseteq \kR_\infty\right) \Big\} +  I_{n,V}\E[J_{n,V}\mid \kF_{n+1}] \\
&=& \sum_{k=0}^{n-1} \sum_{V\cap V'\neq \emptyset} I_{k,V}  I_{k,n,V'}  
\Big\{  \pp_{S_{n+1}-S_k} \left( V\cap V' \nsubseteq \kR_\infty\right) - \pp_{S_n-S_k}\left( V\cap V'\nsubseteq \kR_\infty\right)\Big\}  + r_{n,V},
\end{eqnarray*}
with $|r_{n,V}| \le 2d+1$. Summing up over $V$, we get 
\begin{equation}\label{equ-DM}
\Delta M_n= \sum_{k=0}^{n-1} \sum_{V\cap V'\neq \emptyset} I_{k,V} I_{k,n,V'}
\Big\{  \pp_{S_{n+1}-S_k} \left( V\cap V' \nsubseteq \kR_\infty\right) - \pp_{S_n-S_k}\left( V\cap V'\nsubseteq \kR_\infty\right)\Big\}  + r_n,
\end{equation}
with $|r_n|\le 2^d (2d+1)$. 
 Lemma~\ref{lemLambda} is designed to deal
with the right hand side of \reff{equ-DM}, with the result that
\begin{eqnarray}
\label{Mn1}
\Delta M_n =  \sum_{k=0}^{n-1} \sum_{V\cap V'\neq \emptyset} c(V\cap V')  I_{k,V} I_{k,n,V'} \frac{ \langle S_{n+1}-S_n,S_n-S_k\rangle}{1+ \|S_n-S_k\|^d} + \kO\left(B_n\right),
\end{eqnarray}
with
\begin{equation}\label{def-Bn}
B_n=\sum_{z\in \partial \kR_n} \frac 1{ 1+ \|S_n-z\|^d}.
\end{equation}

\vspace{0.2cm}
\noindent \underline{Step 2.} 
The term $B_n$ of \reff{def-Bn} can be bounded as follows. 
By using first the invariance of the law of the walk by time inversion, we can replace the term $S_n-z$ by $z$. Then we write  
\begin{equation}\label{Bn-1}
\E\big[B^2_n]=
\E\Big[\left(\sum_{z\in \partial \kR_n}
\frac{1}{1+\|z\|^d}\right)^2 \Big]
= \sum_{z,z'\in \Z^d} \frac 1{ (1+ \|z\|^d)(1+\|z'\|^d)}\, \pp(z\in 
\partial \kR_n,\, z'\in \partial \kR_n).
\end{equation}
Then by assuming for instance that $\|z\|\le \|z'\|$ (and $z\neq z'$), and by using \eqref{Green} we obtain
\begin{eqnarray}
\label{Greensbound}
\nonumber 
\pp(z\in \partial \kR_n,\, z'\in \partial \kR_n)&\le & \pp(H_z<\infty,\,  H_{z'}<\infty) \\
 &\le & 2G(0,z)G(z,z') =\kO\left(\frac 1{1+\|z\|^{d-2}\|z'-z\|^{d-2}}\right).
\end{eqnarray}
Therefore 
\begin{eqnarray*}
\E\big[ B^2_n \big] =\kO\left(\sum_{1\le \|z\| < \|z'\| } 
\frac 1{\|z\|^{2d-2}\|z'\|^d \|z'-z\|^{d-2}}\right).
\end{eqnarray*}
Next, we divide the last sum into two parts:
\begin{eqnarray*}
&& \sum_{1\le \|z\|<\|z'\|} \frac 1{\|z\|^{2d-2}\|z'\|^d \|z'-z\|^{d-2}}\\
&=& \sum_{1\le \|z\|<\|z'\|\le 2\|z\|} \frac 1{\|z\|^{2d-2}\|z'\|^d \|z'-z\|^{d-2}}+\sum_{1\le 2\|z\|<\|z'\|} \frac 1{\|z\|^{2d-2}\|z'\|^d \|z'-z\|^{d-2}}\\
&=&\kO\left(\sum_{1\le \|z\|<\|z'\|\le 2\|z\|} \frac 1{\|z\|^{3d-2}\|z'-z\|^{d-2}} + \sum_{1\le 2\|z\|<\|z'\|} \frac 1{\|z\|^{2d-2}\|z'\|^{2d-2}}\right)
= \kO(1).
\end{eqnarray*}
Now it remains to bound the main term in \eqref{Mn1}. 
For two nonempty subsets $U$ and $U'$ of $V_0$, write $U\sim U'$, 
if there exists an isometry of $\Z^d$ sending $U$ onto $U'$. This of course defines an equivalence relation on the subsets of $V_0$, and for any representative $U$ of an equivalence class, we define 
$$\widetilde I_{k,n,U} = \sum_{V\cap V' \sim U} I_{k,V}\, I_{k,n,V'}=\ind\{V_0\cap (\R_n^c-S_k)\sim U \}\,$$
and 
$$H_{n,U} =  \sum_{k=0}^{n-1} \,  \widetilde I_{k,n,U}\,  \frac{ \langle S_{n+1}-S_n,S_n-S_k\rangle}{1+ \|S_n-S_k\|^d}.$$
Note that since the function $c(\cdot)$ is invariant under
isometry, we can rewrite the main term in \eqref{Mn1} as 
\begin{eqnarray}
\label{Mn2}
 \sum_U c(U) \, H_{n,U}.
\end{eqnarray}
Then observe that for any $U$, 
$$
\E[H_{n,U}^2\mid \kF_n] = \left\| \sum_{k=0}^{n-1} \,  \widetilde I_{k,n,U}\,  \frac{ S_n-S_k}{1+ \|S_n-S_k\|^d}   \right\|^2.$$
Moreover, since the law of the walk is invariant under time inversion, and since for any path $S_0,\dots, S_n$, and any $k$, 
the indicator $\widetilde I_{k,n,U}$ is equal to $1$
if and only if it is also equal to $1$ for the reversed 
path $S_n,\dots, S_0$, we get  
\begin{eqnarray}
\label{Mn3}
\E[H_{n,U}^2]=\E[\| \kH_{n,U}\|^2],
\end{eqnarray}
where 
\begin{equation}\label{def-HnU}
\kH_{n,U}:= \sum_{z\in \partial \widetilde \kR_{n,U}} 
\frac{z}{1+ \|z\|^d},
\quad\text{with }\quad
\partial \widetilde \kR_{n,U}
:=  \left\{S_k\ :\ \widetilde I_{k,n,U} =1,\ k\le n-1\right\}.
\end{equation} 
Therefore, we only need to prove that for any $U$,
\begin{eqnarray}
\label{obj.step3}
\E[\| \kH_{n,U}\|^2]=  \left\{ \begin{array}{ll} 
\kO(\log n ) & \textrm{if }d=3\\
\kO(1) &  \textrm{if }d\ge 4. 
\end{array}
\right.
\end{eqnarray}

\vspace{0.2cm}
\noindent \underline{Step 3.} First note that 
\begin{eqnarray}
\label{HnU1}
\E[\| \kH_{n,U}\|^2]=\sum_{z,z'\in \Z^d} \frac{\langle z,z'\rangle}{(1+ \|z\|^d)(1+ \|z'\|^d)} \, \pp\left(z\in \partial \widetilde \kR_{n,U},\, z'\in \partial \widetilde \kR_{n,U}\right).
\end{eqnarray}
In dimension $4$ or larger, \reff{obj.step3} can be established as follows. 
First Cauchy-Schwarz inequality gives for all $z,z'$ 
$$|\langle z,z'\rangle | \, \le\,  \|z\| \, \|z'\|,$$
Then, by using again the standard bound on Green's functions,
that is \eqref{Greensbound}, we get the desired bound
\begin{eqnarray*}
\E[\| \kH_{n,U}\|^2]&=&\kO\left(\sum_{1\le \|z\|< \|z'\|} \|z\|^{3-2d} \, \|z'\|^{1-d} \, \|z-z'\|^{2-d}\right) = \kO(1).
\end{eqnarray*}
We consider now the case $d=3$. 
Since it might be interesting to see what changes in dimension $3$, 
we keep the notation $d$ in all formula as long as possible.    
Note that if $z\in \partial \widetilde \kR_{n,U}$, then 
$\|z\|\le n$ and $H_z$ is finite. Therefore the restriction of the sum in \eqref{HnU1} to the set of $z,z'$ satisfying $\|z\|\le \|z'\|\le 2\|z\|$ is bounded in absolute value by 
\begin{eqnarray}
\label{zz'z}
\sum_z\sum_{z'} \ind_{\{\|z\|\le \|z'\|\le 2\|z\|\le 2n\}} \frac{2\|z\|^2}{(1+ \|z\|^d)^2} \, \pp\left(H_z<\infty,\, H_{z' }<\infty \right).
\end{eqnarray}
Moreover, as we have already recalled, for any $z\neq z'$, with $\|z\|\le \|z'\|$, 
$$\pp(H_z<\infty,\, H_{z'}<\infty) \le 2 G(0,z)G(z,z') = \kO\left( \frac 1{\|z\|^{d-2} \|z-z'\|^{d-2}}\right).$$
Therefore, the sum in \eqref{zz'z} is bounded above (up to some constant) by 
$$ 
\sum_z\sum_{z'}\ind_{\{1\le \|z\| < \|z'\|\le 2\|z\|\le 2n\}} \, \|z\|^{4-3d}  \|z-z'\|^{2-d}= \kO\left( \sum_{1\le \|z\|\le n} \, \|z\|^{6-3d}\right) = \kO (\log n).$$ 
It remains to bound the sum in \eqref{HnU1} restricted to the $z$ and $z'$ satisfying $\|z'\|\ge 2\|z\|$. 
To this end observe that the 
price of visiting $z'$ first is too high. Indeed,   
\[
\begin{split}
\sum_z\sum_{z'}\ind_{\{1\le 2\|z\|\le \|z'\|\le n\}}& \frac{\left| \langle z,z'\rangle\right|}{(1+ \|z\|^d)(1+ \|z'\|^d)} \, \pp\left(H_{z'}<H_z<\infty\right)\\
&= \kO\left(\sum_z\sum_{z'}\ind_{\{1\le 2\|z\|\le \|z'\|\le n\}} \|z\|^{1-d}\|z'\|^{3-2d}\|z-z'\|^{2-d}\right)\\
&=  \kO\left(\sum_z\sum_{z'}\ind_{\{1\le 2\|z\| \le  \|z'\| \le n\}} \|z\|^{1-d}\|z'\|^{5-3d}\right)\\
&= \kO(\log n),
\end{split}
\]
where for the first equality we used in particular Cauchy-Schwarz inequality and again the standard bound on Green's function,  and for the second one, we used that when $\|z'\|\ge 2\|z\|$, we have $\|z'\|\asymp \|z-z'\|$. 
Thus in \eqref{HnU1} we consider the events $\{H_z<H_{z'}\}$. 
We now refine this argument in saying that after $H_{z+V_0}$, and after having left the ball $B(z,\|z\|/2)$, it cost too much to return to $z+V_0$ (and the same fact holds for $z'$). 
Formally, for any $z$, define 
$$\tau_z :=\inf\{k \ge H_{z+V_0}\ :\ S_k\notin B(z,\|z\|/2)\},$$
and 
$$\sigma_z : = \inf\{k \ge \tau_z\ :\  S_k \in z+V_0 \}.$$
Then define the event 
$$E_{z,n,U}:= \left\{z\in  \partial \widetilde \kR_{n,U}\right\} \cap \{\sigma_z=\infty \}.$$
Observe next that if $1\le \|z\|\le  \|z'\|/2$,  
$$ \pp\left(z \in \partial \widetilde \kR_{n,U},\, z' \in \partial \widetilde \kR_{n,U},\, (E_{z,n,U}\cap E_{z',n,U})^c \right) = \kO\left( \frac 1{\|z\|^{2d-4} \|z'\|^{d-2}}\right).$$
Therefore, similar computations as above, show that in \eqref{HnU1}, we can replace the event  
$\{z\text{ and }z'\in \partial \widetilde \kR_{n,U}\}$ by $E_{z,n,U}\cap E_{z',n,U}$. So at this point it remains to bound the (absolute value of the) sum 
\begin{eqnarray}
\label{HnUbis}
\sum_{\|z'\|\ge 2\|z\|} \frac{\langle z,z'\rangle}{(1+ \|z\|^d)(1+ \|z'\|^d)} \, \pp\left(H_z<H_{z'},\, E_{z,n,U},\, E_{z',n,U} \right).
\end{eqnarray}
We now eliminate the time $n$-dependence in 
$E_{z,n,U}$ and $E_{z',n,U}$ by replacing these events
respectively by $E_{z,U}$ and $E_{z',U}$ defined as  
$$E_{z,U}: = \{H_z<\infty\}\cap
\{V_0\cap \{S_0-z,\dots,S_{\tau_z}-z\}^c \sim U\}\cap\{\sigma_z=\infty\}.$$
Note that when $\tau_z\le n$ we have $E_{z,U} = E_{z,n,U}$. This 
latter relation holds in particular when $z$ is visited before $z'$, and $z'$ is visited before 
time $n$. Therefore one has 
$$
E_{z,n,U} \cap  E_{z',n,U} \cap\{ H_z<H_{z'}\}=
E_{z,U} \cap  E_{z',n,U} \cap\{ H_z<H_{z'}\}.
$$ 
In other words in \eqref{HnUbis} one can replace the event $E_{z,n,U}$ by $E_{z,U}$. 
We want now to do the same for $z'$, but the argument is a bit more delicate. First define the symmetric difference of two sets $A$ and 
$B$ as $A\, \Delta\,  B = (A\cap B^c) \cup (A^c\cap B)$.  
Recall that we assume $1\le \|z\|\le  \|z'\|/2$. Let now $k\le n$. By using \eqref{kernel} and \eqref{Green}, we get 
for some positive constants $c$ and $C$ (recall also the definition \eqref{shifthit}), 
\begin{eqnarray*}
\pp_z\left(E_{z',k,U}\, \Delta\,  E_{z',U}\right)&
\le & \pp_z\left(H_{z'}\le k \le H_{z'+V_0}^{(k+1)}<\infty\right) \\
&\le & C\, \E_z\left[ \ind_{\{H_{z'}\le k\}}
 \frac 1{1+\|S_k-z'\|^{d-2}}\right] \\
&\le & C\, \sum_{i=1}^k \pp_z(S_i = z')\, \E\left(\frac 1{1+\|S_{k-i}\|^{d-2}}\right) \\
&\le & C\, \sum_{i=1}^k \frac{e^{-c\|z'\|^2/i}}{i\sqrt i}\frac{1}{
1+\sqrt{k-i}},
\end{eqnarray*}
where for the second and third lines 
we used the Strong Markov Property, and Cauchy-Schwarz 
and \reff{final-green2} for the fourth one. 
The last sum above can be bounded by first separating it into two sums, one with indices $i$ smaller than $k/2$, and the other sum over indices $i\ge k/2$. Then using a comparison with an integral for the first sum, one can see that 
$$
\pp_z\left(E_{z',k,U}\, \Delta\,  E_{z',U}\right)\le C
 \frac 1{\|z'\| \sqrt k} \, e^{-c\|z'\|^2/(2k)}.$$
In particular 
$$\sup_{k\ge 1} \, \pp_z\left(E_{z',k,U} \, \Delta\,  E_{z',U}\right) = \kO\left(\frac 1{\|z'\|^2}\right).$$
Then it follows by using again the Markov property and \eqref{Green}, that 
\begin{eqnarray*}
\pp\left(E_{z,U},\, E_{z',n,U}\, \Delta\,  E_{z',U},\, H_z<H_{z'}\right)& \le & \pp\left(E_{z',n,U} \, \Delta\, E_{z',U},\, H_z<H_{z'}\right)\\
&=& \sum_{k\le n} \pp(H_z=n-k) \,\pp_z\left(E_{z',k,U} \, \Delta\,  E_{z',U}\right)\\
&\le &  \pp(H_z\le n)
\sup_{k\ge 1} \, \pp_z\left(E_{z',k,U} \, \Delta\,  E_{z',U}\right)\\
&=&\kO\left(\frac 1{\|z\|^{d-2}}\times \frac 1{ \|z'\|^2}\right).
\end{eqnarray*}
In conclusion, one can indeed replace the event $E_{z',n,U}$ by $E_{z',U}$ in \eqref{HnUbis}. 
Now in the remaining sum, we gather together the pairs $(z,z')$ and $(z,-z')$, and we get, using Cauchy-Schwarz again,  
\begin{eqnarray}
\label{difference}
&&\left|  \sum_{1\le 2\|z\|\le \|z'\|\le n} \frac{ \langle z,z'\rangle}{(1+ \|z\|^d)(1+ \|z'\|^d)} \, \pp\left(E_{z,U},\, E_{z',U},\, H_z<H_{z'}\right)\right|  \\
\nonumber &\le & \sum_{1\le 2\|z\|\le \|z'\|\le n} \frac 2{\|z\|^{d-1} \|z'\|^{d-1}} \,\Big| \pp\left(E_{z,U},\, E_{z',U},\, H_z<H_{z'}\right) - \pp\left(E_{z,U},\, E_{-z',U},\,  H_z<H_{-z'}\right) \Big| .
\end{eqnarray}
Then for any $1\le \|z\|\le \|z'\|/2$, we have 
\begin{eqnarray}
\label{HzV0}
\nonumber &&\pp\left(E_{z,U},\, E_{z',U},\, H_z<H_{z'}\right)\\
&=& \sum_{y\in \overline \partial B(z,\|z\|/2)} \pp\left(E^*_{z,U}, \, H_z<H_{z'},\, S_{\tau_z}=y\right)\,  \pp_y\left(H_{z+V_0}=\infty,\,  E_{z',U}\right), 
\end{eqnarray}
where by $\overline \partial B(z,\|z\|/2)$ we denote the external boundary of $B(z,\|z\|/2)$, and where  
$$E^*_{z,U}:= \{V_0\cap \{S_0-z,\dots,S_{\tau_z}-z\}^c \sim U\}.$$
Now for any $y\in \overline \partial B(z,\|z\|/2)$, and $\|z'\|\ge 2\|z\|$, by using again \eqref{Green} we get
\begin{eqnarray}
\label{HzV02}
\pp_y\left(H_{z+V_0}=\infty,\, E_{z',U}\right)= \pp_y\left(E_{z',U}\right) - \kO\left( \frac 1{\|z\|^{d-2} \|z'\|^{d-2}}\right).
\end{eqnarray}
Moreover, the same argument as in the proof of Lemma \ref{lemLambda}, shows that if $y$ and $y'$ are neighbors,  
$$ \pp_y\left(E_{z',U}\right)= \pp_{y'}\left(E_{z',U}\right) + \kO\left(\frac 1{\|z'-y\|^{d-1}}\right).$$
Therefore if $1\le \|z\|\le \|z'\|/2$ and  $y \in \overline \partial B(z,\|z\|/2)$ we get 
\begin{eqnarray}
\label{HzV03}
\pp_y\left(E_{z',U}\right) =  \pp_0\left(E_{z',U}\right) +  \kO\left(\frac{\|y\|}{\|z'-y\|^{d-1}}\right)=\pp_0\left(E_{z',U}\right) +  \kO\left(\frac{\|z\|}{\|z'\|^{d-1}}\right).
\end{eqnarray}
On the other hand, by symmetry, for any $U$, 
$$\pp_0\left(E_{z',U}\right) = \pp_0\left(E_{-z',U}\right).$$
By combining this with \eqref{difference}, \eqref{HzV0}, \eqref{HzV02}  and \eqref{HzV03},  
we obtain \eqref{obj.step3} and conclude the proof of \eqref{deltaMn}.

\section{Lower Bound on the Variance}\label{sec-LBvar}
In this section, we prove Proposition~\ref{prop-vard3}. The proof is inspired by the proof of Theorem 4.11 in \cite{MS}, where the authors use lazyness of the walk. Here, since the walk we consider is not lazy, we use instead the notion of double backtracks. 
We say that the simple random walk makes a double
backtrack at time $n$, when $S_{n+1}=S_{n-1}$ and $S_{n+2}=S_n$. When this happens 
the range (and its boundary) remain constant during steps
$\{n+1,n+2\}$. With this observation in mind,
a lower bound on the variance is obtained
as we decompose the simple random walk into two independent processes:
a clock process counting the number of double-backtracks (at even times), and
a trajectory without double-backtrack (at even times).

\subsection{Clock Process}
We construct by induction 
a no-double backtrack walk $(\widetilde S_n,\ n\in \N)$.
First, $\widetilde S_0=0$, and $\widetilde S_1$ and $\widetilde S_2-\widetilde S_1$
are chosen uniformly
at random among the elements of $V_0$ (the set of neighbors of the origin). 
Next, assume that $\widetilde S_k$ has been defined for all $k\le 2n$, for some $n\ge 1$.
Let $\kN_2=\{(x,y):\ x\sim 0$ and $y\sim x\}$ and choose $(X,Y)$ uniformly
at random in $\kN_2\setminus\{\big(\widetilde S_{2n-1}-
\widetilde S_{2n},0\big)\}$. Then set
\begin{equation}\label{skel-1}
\widetilde S_{2n+1}=\widetilde S_{2n}+X\quad\text{and}\quad
\widetilde S_{2n+2}=\widetilde S_{2n}+Y.
\end{equation}
Thus, the walk $\widetilde S$ makes no double-backtrack at even times. Note that
by sampling uniformly in the whole of $\kN_2$ 
we would have generated a simple random walk (SRW).
Now, to build a SRW out of $\widetilde S$, it is enough to add at
each even time a geometric number of double-backtracks. The geometric
law is given by 
\begin{equation}\label{clock-1}
 \pp(\xi=k)=(1-p)p^k \quad \text{for all }k\ge 0, 
\end{equation}
with $p=1/(2d)^2$. Note that the mean of $\xi$ is equal to $p/(1-p)$.
Now, consider a sequence $(\xi_n,\ n\ge 1)$ of i.i.d. random variables distributed like $\xi$  
and independent of $\widetilde S$.
Then define 
\begin{equation}\label{clock-2}
\widetilde N_0=\widetilde N_1=0\quad \text{and}\quad \widetilde N_k:=\sum_{i=1}^{[k/2]} \xi_i \quad \text{for }k\ge 2.
\end{equation}
A SRW can be built out from $\widetilde S$ and $\widetilde N$ as follows.
First, $S_i=\widetilde S_i$ for $i=0,1,2$. Then, for any integer $k\ge 1$
\[
S_{2i-1}=\widetilde S_{2k-1}\quad\text{and}\quad 
S_{2i}=\widetilde S_{2k} \quad \text{for all } i\in [k+\widetilde N_{2(k-1)} ,k+\widetilde N_{2k}].
\]
This implies that if $\widetilde \kR$ is the range of $\widetilde S$ and $\partial \widetilde \kR$ its boundary, then 
for any integer $k$
\begin{equation}\label{clock-main}
\kR_{k+2\widetilde N_k}=\widetilde \kR_k \qquad \text{and} \qquad \partial\kR_{k+2\widetilde N_k}=\partial\widetilde \kR_k.
\end{equation}

\subsection{A Law of Large Numbers and some consequences}
Recall that Okada \cite{Ok1} proved a law of large numbers for $|\partial \kR_n|$, 
see \eqref{okada-2}, and call $\nu_d$ the limit of $|\partial \kR_n|/n$. 
Since $\widetilde N_n/n$ also converges almost surely toward $p/[2(1-p)]$, we deduce from \eqref{clock-main} that  
\begin{equation}
\label{LLNtilde}
\frac{|\partial \widetilde \kR_n|}{n}\ \longrightarrow \  \frac{\nu_d}{1-p} \qquad \text{almost surely}.
\end{equation} 
Let us show now another useful property.
We claim that for any $\alpha>0$,
\begin{equation}\label{inter.tildeR}
\lim_{r\to\infty}
\pp\big(|(\R'_\infty)^+\cap \widetilde \kR_r^+|\ge \alpha r\big)=0,
\end{equation}
where $\R'_\infty$ is the total range of another  
simple random walk independent of $\widetilde \R$.  
To see this recall that the process $(\kR_n)$ is increasing,
and therefore using \reff{clock-main} one deduce 
$$
\pp\big(|(\kR'_\infty)^+\cap \widetilde \kR_r^+|\ge \alpha r\big) 
\le  \pp\big(\widetilde N_r\ge \frac{p}{1-p}r\big)
+\pp\big(|(\kR'_\infty)^+\cap \kR_{Cr}^+|\ge \alpha r\big), 
$$
with $C=2p/(1-p)+1$. 
The first term on the right-hand side goes to $0$, in virtue of the law of large numbers satisfied by $\widetilde N$, and the second one also  as can be seen using 
Markov's inequality and 
the estimate:
$$
\E[|(\kR'_\infty)^+\cap \kR_{Cr}^+|]
\le \sum_{x,y\in \overline V_0}\sum_{z\in \Z^d} 
G(0,z+x)G_{Cr}(0,z+y) = \kO(\sqrt{r} \log r),$$
which follows from \eqref{Green} and \cite[Theorem 3.6]{LG}.

A consequence of \reff{inter.tildeR} is the following.
Define $c=\nu_d/[2(1-p)]$. We have that for $k$ large enough, any 
$t\ge 1$, and $r\ge \sqrt k$
\begin{equation}\label{lln-5}
\pp\big(|\partial \widetilde \R_k| \ge t\big)\ge \frac{1}{2}
\quad \Longrightarrow \quad \pp\big(|\partial \widetilde \R_{k+r}| 
\ge t+c r \big)\ge \frac{1}{4},
\end{equation}
and also
\begin{equation}\label{lln-6}
\pp\big(|\partial \widetilde \R_k| \le t\big)\ge \frac{1}{2}
\quad \Longrightarrow \quad 
\pp\big(|\partial \widetilde \R_{k-r}| \le t-c r  \big)\ge \frac{1}{4}.
\end{equation} 
To see this first note that the set-inequality
\reff{boundary-ineq} holds as well for $\widetilde \R$. Hence, with evident notation
\begin{equation}\label{lln-7}
|\partial \widetilde \R_{k+r}| \ge 
|\partial \widetilde \R_r|+|\partial \widetilde \R(r,k+r)| -
2|\widetilde \R^+(r,r+k)\cap\widetilde \R_r^+|.
\end{equation}
Now observe that the last intersection term is stochastically dominated by $|(\R'_\infty)^+\cap\widetilde \R_r^+|$, with $\R'_\infty$ a copy of $\kR_\infty$, independent of $\widetilde \R_r^+$. Therefore,  \eqref{LLNtilde}, \eqref{inter.tildeR} and \reff{lln-7} 
immediately give \eqref{lln-5} and \eqref{lln-6}. 

\subsection{Lower Bound}
First, by using \reff{prop-dyad-1}, there
is a positive constant $C_0>2d$,  such that
\begin{equation}\label{dev.esp}
\left|\E[|\partial \R_n|] - \nu_d n\right|\le C_0  \sqrt n\quad \text{for all }n\ge 1.
\end{equation}
Take $k_n$ to be the integer part of $(1-p)n$. We have
either of the two possibilities
\begin{equation}\label{lbv-1}
(i)\quad \pp\big( |\partial \widetilde \kR_{k_n}| \le \nu_d n\big)\ge \frac{1}{2} 
\quad\text{or}\quad 
(ii)\quad \pp\big(|\partial \widetilde \kR_{k_n}|  \ge \nu_d n\big)\ge \frac{1}{2}.
\end{equation}
Assume for instance that (i) holds, and note that (ii) would be treated symmetrically. 
Define, $i_n=[(1-p)(n-A\sqrt n)]$, with $A=3C_0/(c(1-p))$, and note that using \reff{lln-6}
\begin{equation}\label{lbv-2}
\pp\big(|\partial \widetilde \kR_{i_n}|\le  \nu_d n- 3C_0\sqrt n\big)
)\ge \frac{1}{4},
\end{equation}
for $n$ large enough. 
Now set 
\[
\kB_n=\left\{ \frac{2\widetilde N_{i_n}-2\E[\widetilde N_{i_n}]}{\sqrt n}\in [A+1,A+2]\right\}.
\]
Note that
there is a constant $c_A>0$, such that for all $n$ large enough
\begin{equation}\label{lbv-3}
\pp(\kB_n)\ge c_A.
\end{equation}
Moreover, by construction, 
\begin{equation}
\label{Bninclusion}
\kB_n\ \subset \ \left\{ i_n+2\widetilde N_{i_n} \in [n,n+3\sqrt n]\right\}.
\end{equation}
Now using the independence of $\widetilde N$ and $\partial \widetilde \kR$, \eqref{clock-main}, \eqref{lbv-2}, \eqref{lbv-3} and \eqref{Bninclusion}, we deduce that   
$$
\pp\big(\exists m\in \{0,\dots, 3\sqrt n\}\ :\ |\partial \R_{n+m}| \le \nu_d n - 3C_0\sqrt n\big)
\ge \frac{c_A}{4}.
$$
Then one can use the deterministic bound:  
$$
|\partial \kR_n| \le  |\partial \kR_{n+m}| + 2dm,
$$
which holds for all $n\ge 1$ and $m\ge 0$. 
This gives 
$$\pp(|\partial \kR_n|\le \nu_d n -2C_0\sqrt n)\ge \frac{c_A}4,$$
which implies that $\var(|\partial \kR_n|)/n\ge C_0^2 c_A/4>0$, using \eqref{dev.esp}.

\section{On Le Gall's decomposition}\label{sec-legall}
In this Section, we establish Proposition~\ref{prop-legall} and 
Theorem~\ref{prop-dyadic}.

\subsection{Mean and Variance}
Inequality \reff{boundary-ineq} holds since
\[
z\in \partial \R(0,n)\bs \big(S_n+\R(n,n+m)\big)^+\cup
\partial(S_n+ \R(n,n+m))\bs \R^+(0,n)
\Longrightarrow z\in \partial \R(0,n+m).
\]
Define 
$$X(i,j)=|\partial \R(i,j)|\quad \textrm{and} \quad \overline X(i,j)=X(i,j)-\E[X(i,j)].$$
Observe that in \reff{boundary-ineq} the deviation from linearity 
is written in terms of an intersection of two independent
range-like sets. This emphasizes the similarity between
range and boundary of the range.
Now \reff{boundary-ineq} implies
the same inequalities for the expectation.
\begin{equation}\label{expectation-ineq}
0\ge \E[X(0,n+m)]-\big(\E[X(0,n)]+\E[X(n,n+m)]\big)\ge -\E[Z(n,m)].
\end{equation}
Combining \reff{boundary-ineq} and \reff{expectation-ineq}, we obtain
our key (and simple) estimates
\begin{equation}\label{main-ineq}
|\overline X(0,n+m)-\big(\overline X(0,n)+ \overline X(n,n+m)\big)|\le
\max\big(Z(n,m),\E[Z(n,m)]\big).
\end{equation}
If $\|X\|_p=(\E[X^p])^{1/p}$, then using the triangle inequality,
we obtain for any $p>0$,
\begin{equation}\label{norm-ineq}
\big|\| \overline X(0,n+m)\|_p- \|\overline X(0,n)+ \overline X(n,n+m)\|_p
\big|\le \|Z(n,m)\|_p+\|Z(n,m)\|_1.
\end{equation}
The deviation from linearity of the centered $p$-th moment
will then depend on the $p$-th moment of $Z(n,m)$.
We invoke now Hammersley's Lemma \cite{HA}, which extends the classical subadditivity
argument in a useful manner.
\bl{lem-ham}[Hammersley] Let $(a_n)$, $(b_n)$, and $(b'_n)$ be
sequences such that 
\begin{equation}\label{hyp-hamine}
a_n+a_m-b'_{n+m}\le a_{m+n}\le a_n+a_m+b_{n+m} 
\qquad \textrm{for all $m$ and $n$}.
\end{equation}
Assume also that
the sequences $(b_n)$ and $(b'_n)$ are positive and non-decreasing, and
satisfy
\begin{equation}\label{hyp-ham}
\sum_{n>0} \frac{b_n+b'_n}{n(n+1)}<\infty.
\end{equation}
Then, the limit of $a_n/n$ exists, and
\begin{equation}\label{conc-ham}
-\frac{b'_n}{n}+4
\sum_{k>2n} \frac{b'_k}{k(k+1)}\ge
\frac{a_n}{n}-\lim_{k\to\infty} \frac{a_k}{k}\ge+\frac{b_n}{n}-4
\sum_{k>2n} \frac{b_k}{k(k+1)}.
\end{equation}
\el
\noindent We obtain now the following moment estimate.
\bl{lem-moment}
For any integer $k$, there is a constant $C_k$ such that
for any $n,m$ integers,
\begin{equation}\label{ineq-moment}
\E\big[Z^k(n,m)\big]\le C_k
\big(\psi^k_d(n)\psi^k_d(m)\big)^{1/2}.
\end{equation}
Recall that $\psi_d$ is defined in \reff{def-psi}.
\el
\begin{proof}
Observe that $Z(n,m)$ is bounded as follows.
\begin{eqnarray*}\label{Z-1}
\nonumber Z(n,m)& \le& 2\sum_{z\in \Z^d}
\ind\{z\in \R^+(0,n)\cap \big(S_n+\R^+(n,n+m)\big)\}\\
\nonumber &\le & 2\sum_{z\in \Z^d}
\ind\{z\in \overleftarrow{\R}^+(0,n)\cap \R^+(n,n+m)\}\\
&\le & 2\sum_{x\in \overline V_0}\sum_{y\in \overline V_0}\sum_{z\in \Z^d}
\ind\{z+x\in  \overleftarrow{\R}(0,n),\ z+y\in \R(n,n+m)\}.
\end{eqnarray*}
We now take the expectation of the $k$-th power, and
use the independence of $ \overleftarrow{\R}(0,n)$ and
$ \R(n,n+m)$. Then (recalling the defintion \eqref{def.Gn} of $G_n$ and using \eqref{bubble-G} for the last inequality) 
\begin{equation*}\label{Z-2}
\begin{split}
\E[Z^k &(n,m)]\le 2^k\sum_{x_1,y_1\in \overline V_0}\!\dots\!
\sum_{x_k,y_k\in \overline V_0}\sum_{z_1,\dots,z_k}\E\left[
\prod_{i=1}^k \ind\{z_i+x_i\in \overleftarrow{\R}(0,n)\}
\ind\{z_i+y_i\in \R(n,n+m)\}\right]\\
\le & \  2^k\sum_{x_1,y_1\in \overline V_0}\!\dots\!
\sum_{x_k,y_k\in \overline V_0}\sum_{z_1,\dots,z_k}
\pp\big(H_{z_i+x_i}<n,\ \forall i=1,\dots,k\big)
\pp\big(H_{z_i+y_i}<m,\ \forall i=1,\dots,k\big)\\
\le & \ 2^k |\overline V_0|^{2k}
\Big(\sum_{z_1,\dots,z_k}
\pp\big(H_{z_i}<n,\ \forall i=1,\dots,k\big)^2 \Big)^{1/2}
\Big(\sum_{z_1,\dots,z_k}
\pp\big(H_{z_i}<m,\ \forall i=1,\dots,k\big)^2 \Big)^{1/2}\\
\le & \ 2^k |\overline V_0|^{2k}k!
\Big(\sum_{z_1,\dots,z_k} G^2_n(0,z_1)\dots G^2_n(z_{k-1},z_k)\Big)^{1/2}
\Big(\sum_{z_1,\dots,z_k} G^2_{m}(0,z_1)\dots
G^2_{m}(z_{k-1},z_k)\Big)^{1/2}\\
\le & \ C_k \Big(\psi_d^k(n)\psi_d^k(m)\Big)^{1/2},
\end{split}
\end{equation*}
which concludes the proof. 
\end{proof}
Henceforth, and for simplicity, we think of $\psi_d$ of \reff{def-psi}
rather as $\psi_3(n)=\kO(\sqrt n)$, $\psi_4(n)=\kO(\log(n))$ and for
$d\ge 5$, $\psi_d(n)=\kO(1)$ (in other words, we aggregate in $\psi_d$
innocuous constants).
As an immediate consequence of \reff{expectation-ineq} and
Lemma~\ref{lem-moment}, we obtain for any $n,m\in \N$, 
\begin{equation}\label{dyad-1}
\E[|\partial \R_n|]+\E[|\partial \R_{m}|]
-\max(\psi_d(n),\psi_d(m))\le \E[|\partial \R_{n+m}|] \le
\E[|\partial \R_n|]+\E[|\partial \R_{m}|].
\end{equation}
The inequalities of \reff{dyad-1} and Hammersley's Lemma imply that
the limit of $\E[|\partial \R_n|]/n$ exists
and it yields \reff{prop-dyad-1} of Theorem~\ref{prop-dyadic}. 

\paragraph{Variance of $X(0,n)$.}
Let us write \reff{norm-ineq} for $p=2$
\begin{equation}\label{dyad-var}
\big| \|\overline X(0,n+m)\|_2- \|\overline X(0,n)+\overline X(n,n+m)\|_2\big|\le
2 \|Z(n,m)\|_2.
\end{equation}
Now, the independence of $\overline X(0,n)$ and $\overline X(n,n+m)$ gives
\begin{equation}\label{dyad-indep}
 \|\overline X(0,n)+\overline X(n,n+m)\|_2^2= \|\overline X(0,n)\|_2^2+
 \|\overline X(0,m)\|_2^2.
\end{equation}
By taking squares on both sides of \reff{dyad-var} and using
\reff{dyad-indep}, we obtain
\begin{equation}\label{dyad-detail1}
\begin{split}
\|\overline X(0,n+m)\|_2^2\le& \|\overline X(0,n)\|_2^2+
\|\overline X(0,m)\|_2^2+4 \|\overline X(0,n)
+\overline X(n,n+m)\|_2\|Z(n,m)\|_2\\
&\qquad +4\|Z(n,m)\|_2^2,
\end{split}
\end{equation}
and
\begin{equation}\label{dyad-detail2}
\begin{split}
\|\overline X(0,n)\|_2^2+\|\overline X(0,m)\|_2^2\le&
\|\overline X(0,n+m)\|_2^2+4\|\overline X(0,n+m)\|_2\|Z(n,m)\|_2\\
&\qquad+ 4\|Z(n,m)\|_2^2.
\end{split}
\end{equation}
Now define for $\ell \ge 1$, 
$$A_\ell := \sup_{2^\ell < i\le 2^{\ell +1}} \|\overline X(0,i)\|_2^2.$$
Next, using \eqref{dyad-indep}, \reff{dyad-detail1} and
Lemma~\ref{lem-moment} with $k=2$, we deduce that for any $\ell\ge 1$ and 
$\varepsilon>0$ (using also the inequality $2ab \le \varepsilon a^2+ b^2/\varepsilon$), 
\begin{equation}\label{dyadic-4}
A_{\ell +1} \le (1+\varepsilon)
2A_\ell +(1+\frac{1}{\varepsilon})\psi^2_d(2^\ell).
\end{equation}
We iterate this inequality $L$ times to obtain for some constant $C$ independent of $L$
\begin{equation}\label{dyadic-5}
\begin{split}
A_L \le & \ C(1+\varepsilon) ^L 2^L+
C(1+\frac{1}{\varepsilon})\sum_{\ell=1}^L(1+\varepsilon)^{\ell-1}2^{\ell-1}\psi^2_d(2^{L-\ell})\\
\le &\ C \, L^2 \, 2^L  \qquad
\text{when we choose}\qquad \varepsilon=\frac{1}{L}.
\end{split}
\end{equation}
Then we use the rough bound of \reff{dyadic-5} as an a priori
bound for the upper and lower bounds respectively \reff{dyad-detail1}
and \reff{dyad-detail2} for the sequence $a_n=\var\big( X(0,n)\big)$,
in order to apply Hammersley's Lemma with 
$b_n=b'_n=\sqrt{n}\log n \times \psi_d(n)$.
In dimension four or more
we do fulfill the hypotheses of Hammersley's Lemma,
which in turn produces the improved bound $\var(|\partial \R_n|)\le
Cn$, and then again we can use Hammersley's 
Lemma with a smaller $b_n=b'_n=\sqrt n \psi_d(n)$ which eventually 
yields \reff{prop-dyad-2} of Theorem~\ref{prop-dyadic}. The fact that the limit of the normalized variance is positive follows from 
Proposition \ref{prop-vard3}. 

\subsection{Central Limit Theorem}
The principle of Le Gall's decomposition is to repeat
dividing each strand into smaller and smaller pieces producing
independent boundaries of {\it shifted} ranges. For two reals
$s,t$ let $[s],[t]$ be their integer parts and define
$X(s,t)=X([s],[t])$. For $\ell$ and $k$ integer,
let $X^{(\ell)}_{k,n}=X((k-1)n/2^\ell,kn/2^\ell)$.
Let also $Z^{(\ell)}_{k,n}$ be the functional of the two strands
obtained by dividing the $k$-th strand after $\ell-1$ divisions.
In other words, as in \reff{def-Z} (but without translating here) let
\[
Z^{(\ell)}_{k,n}=|U\cap \widetilde U^+|+|U^+\cap \widetilde U|,
\]
with
\[
U:=\{S_{[(k-1)\frac{n}{2^\ell}]},\dots,S_{[(2k-1)\frac{n}{2^{\ell+1}}]}\},
\quad\text{and}\quad
\widetilde U
:=\{S_{[(2k-1)\frac{n}{2^{\ell+1}}]},\dots,S_{[k\frac{n}{2^\ell}]}\}.
\]
Thus, after $L$ divisions, with $2^L\le n$, we get 
\begin{equation*}\label{legall-key}
\sum_{i=1}^{2^L} X^{(L)}_{i,n}-\sum_{\ell=1}^L\sum_{i=1}^{2^{\ell-1}}
Z^{(\ell)}_{i,n}\le X(0,n)\le \sum_{i=1}^{2^L} X^{(L)}_{i,n}.
\end{equation*}
The key point is that $\{X^{(L)}_{i,n},\ i=1,\dots,2^L\}$
are independent, and have the
same law as $X(0,n/2^L)$ or $X(0,n/2^L+1)$. 
Now, we define the (nonnegative)  
error term $\kE(n)$ as
\begin{equation*}\label{def-E}
X(0,n)=\sum_{i=1}^{2^L} X^{(L)}_{i,n} -\kE(n),
\end{equation*}
and \reff{boundary-ineq} and \eqref{prop-dyad-1} imply that
\begin{equation*}\label{estimate-E}
\E[\kE(n)]\le \sum_{\ell=1}^L 2^\ell \psi_d(\frac{n}{2^\ell}).
\end{equation*}
Note that, in dimension $d\ge 4$,
we can choose $L$ growing to infinity with $n$, and
such that $\E[\kE(n)]/\sqrt n$ goes to 0: for instance
$2^L=\sqrt n/\log^2(n)$. Therefore, for such choice of $L$, 
it suffices to prove the Central Limit Theorem 
for the sum $\sum_{i=1}^{2^L} X^{(L)}_{i,n}$. 
Our strategy is to apply the Lindeberg-Feller 
triangular array Theorem, that we recall for convenience (see for instance \cite[Theorem~3.4.5]{Durrett} for a proof).  
\begin{theorem}[Lindeberg-Feller]
\label{thm:lind}
For each integer $N$ let $(X_{N,i}: \, 1\leq i\leq N)$ 
be a collection of independent random variables with zero mean. Suppose that the following two conditions are satisfied
\newline
{\rm{(i)}} $\sum_{i=1}^{N}\E [X_{N,i}^2] \to \sigma^2>0$ as $N\to \infty$ and 
\newline
{\rm{(ii)}} $\sum_{i=1}^{N}\E[(X_{N,i})^2\ind\{|X_{N,i}|>\varepsilon\}] \to 0$ as $N\to \infty$ for all $\varepsilon>0$.
\newline
Then, $S_N=X_{N,1}+\ldots + X_{N,N} 
\Longrightarrow \sigma \kN(0,1)$ as $N\to \infty$.
\end{theorem}
We apply Lindeberg-Feller's Theorem with $N=2^L$ and 
$X_{N,i} = \overline X^L_{i,n}/\sqrt{n}$.
The condition (i) was proved in the previous subsection. 
The condition (ii) is usually called Lindeberg's condition. To check (ii),
we estimate the fourth moment of $\overline X(0,n)$, and 
as was noticed by Le Gall \cite[Remark (iii) p.503]{LG}, 
this is achieved
using the previous decomposition and a sub-additivity argument. More precisely, using \reff{norm-ineq} with $p=4$, we have
\begin{equation*}\label{fourth-1}
\begin{split}
\|\overline X(0,n+m)\|_4\le&\  \Big(
\big( \E[\overline X^4(0,n)]+6 \E[\overline X^2(0,n)]\E[\overline X^2(0,m)]+
\E[\overline X^4(0,m)]\big)\Big)^{1/4}\\
&\qquad +\sqrt{\psi_d(n)\psi_d(m)}.
\end{split}
\end{equation*}
Thus, if we define for $\ell \ge 1$, 
$$A'_\ell := \sup_{2^\ell < i\le 2^{\ell +1}} \|\overline X(0,i)\|_4,$$
we obtain (using also that $(a+b)^{1/4}\le a^{1/4} + b^{1/4}$ for any $a$, $b$),  
\begin{equation*}\label{dyadic-6}
\begin{split}
A'_{\ell+1} \le & \ (2(A'_\ell)^4+6A_\ell^2)^{1/4}+\psi_d(2^\ell)\\
\le &\  2^{1/4} A'_\ell+ 6^{1/4}A_\ell^{1/2} +\psi_d(2^\ell).
\end{split}
\end{equation*}
Define then $A''_\ell= \sup_{2^\ell < i\le 2^{\ell +1}} \|\overline X(0,i)\|_4/2^{\ell/2}$, and recall that $A_\ell \le C_d 2^\ell$, 
for some constant $C_d>0$, in dimension four and larger. Therefore if $d\ge 4$,  
\begin{equation*}\label{induction-dyadic}
A''_{\ell +1 }\le 
\frac{2^{1/4}}{2^{1/2}} A''_{\ell}+6^{1/4}C_d 
+ \frac{\psi_d(2^\ell)}{2^{(\ell +1)/2}}.
\end{equation*}
This recursive inequality implies that
$(A''_\ell)$ is bounded, as well as $n\mapsto \|\overline X(0,n)\|^2_4/n$. 
We then deduce that Lindeberg's condition is satisfied, and 
the Central Limit Theorem holds for $X(0,n)$. 
\vskip 1,cm
\appendix
\begin{center}{\bf Appendix}
\end{center}

\section{Estimates on Ranges}\label{sec-ranges}
In this section, we prove Proposition \ref{propRn}.
We first introduce some other range-like sets allowing us 
to use the approach of Jain and
Pruitt \cite{JP}. 
Recall that the sets $\kR_{n,V}$ are disjoint, 
and for $U\subset V_0$, define
\begin{eqnarray}
\label{link-JP}
\overline{\kR}_{n,U}:=\bigcup_{V\supset U} \kR_{n,V}=\{S_k\ :\ S_k \notin \R_{k-1} \text{ and }S_i\notin (S_k+U),\, i\le k-1,\, 1\le k\le n\}.
\end{eqnarray}
Next for $U\subset  V_0$, define
\begin{eqnarray*}
\alpha(U)=\big| \overline{\kR}_{n,U}\big| - \E\big(\big| \overline{\kR}_{n,U}\big|\big)
\qquad \text{and}\qquad
\beta(U)=\big| \kR_{n,U}\big| -\E\big(\big| \kR_{n,U}\big|\big).
\end{eqnarray*}
The definition \eqref{link-JP} yields
\begin{eqnarray*}
\alpha(U)=\sum_{V\supset U} \beta(V), 
\end{eqnarray*}
and this relation is inverted as follows: 
\begin{eqnarray*}
\label{inverse-ab}
\beta(V)=\sum_{U\supset V} (-1)^{|U\backslash V|}\, \alpha(U).
\end{eqnarray*}
As a consequence, for $V\subset V_0$, 
\begin{eqnarray*}
\textrm{Var}(|\kR_{n,V} |)=\E\left(\beta^2(V)\right)\le 2^{|V_0\backslash V|}\, 
\sum_{U\supset V} \E\left(\alpha^2(U)\right).
\end{eqnarray*}
We will see below that each $\overline{\kR}_{n,V}$ has the same law as a range-like
functional that Jain and Pruitt analyze by using a last passage decomposition, after introducing
some new variables. But let us give more details now. So first, we fix some $V\subset V_0$, and for $n\in \N$, set  
$Z_n^n=1$, and 
\begin{eqnarray*}
\begin{array}{llll}
Z_i &=& {\bf 1}\left(\{S_{i+k}\not\in (S_i+\overline V )\quad  \forall k\ge 1\}\right) &  \forall i\in \N,\\
Z_i^n &=& {\bf 1}\left(\{S_{i+k}\not\in (S_i+\overline V) \quad \forall k=1,\dots,n-i\}\right) & \forall i<n \\
W_i^n &=&Z^n_i-Z_i  & \forall i\le n, 
\end{array}
\end{eqnarray*}
where 
$$\overline V= V\cup \{0\}.$$
A key point in this decomposition is that $Z_n$ and $Z^n_i$ are independent.
Now, define
\begin{eqnarray}
\label{def-JPR}
\underline \kR_{n,V}=\{S_k:\ S_i\not\in S_k+\overline V,\ n\ge i>k,\ 0\le k< n\},
\quad\text{and}\quad |\underline \kR_{n,V}|=\sum_{i=0}^{n-1} Z^n_i.
\end{eqnarray}
Since the increments are symmetric and independent,
$|\overline{\kR}_{n,V}|$ and $|\underline \kR_{n,V}|$
are equal in law.
Now, equality \eqref{def-JPR} reads as
\begin{eqnarray*}
|\underline \kR_{n,V}|=\sum_{i=0}^{n-1} Z_i+\sum_{i=0}^{n-1} W^n_i.
\end{eqnarray*}
Now using that Var$(|\overline{\kR}_{n,V}|)\le \E[(
|\overline{\kR}_{n,V}|-\sum_{i\le n-1} \E[Z_i])^2]$, and that
$(a+b)^2\le 2(a^2+b^2)$ we obtain
\begin{eqnarray}
\label{JP-var}
\textrm{Var}(|\underline \kR_{n,V}|) \le 
2\sum_{i=1}^{n-1} \textrm{Var}(Z_i)+
4\sum_{j=1}^{n-1}\sum_{i=0}^{j-1}\textrm{Cov}(Z_i,Z_j)+ 
4\sum_{j=1}^{n-1}\sum_{i=0}^j \E\left(W^n_iW^n_j\right).
\end{eqnarray}
Next for $i<j< n$, we have (recall the definition \eqref{shifthit})  
\begin{eqnarray*}
\E\left(W^n_i\, W^n_j\right) &=& \pp\left(n<H^{(i+1)}_{S_i+\overline V}<\infty,\, n<H^{(j+1)}_{S_j+\overline V}<\infty \right)\\
&=& \sum_{x\notin \overline V}\pp (S_{j-i}=x,\ H_{\overline V}>j-i)\, \pp_x\big(
n-j<H_{\overline V}<\infty,\ n-j<H_{x+\overline V}<\infty\big)\\
&\le &  \sum_{x\not\in V}\pp (S_{j-i}=x,\ H_{\overline V}>j-i)\, \pp_x\big(
n-j<H_{\overline V}<\infty,\ n-j<H_{x+\overline V}<\infty\big),
\end{eqnarray*}
where for the second equality we just used the Markov property 
and translation invariance of the walk. The last inequality
is written to cover as well the case $i=j$. Therefore,
\begin{eqnarray*}
\sum_{i=0}^j \E\left(W^n_iW^n_j\right)&\le & 
\sum_{x\not\in V} G_j(0,x) \, \pp_x\big(H_{\overline V}<\infty,\ n-j<H_{x+\overline V}<\infty\big) \\
& \le &\sum_{y,z\in \overline V}\, \sum_{x\notin V}\,  G_j(0,x)\, 
\pp_x\big(H_y<\infty,\ n-j<H_{x+z}<\infty\big). 
\end{eqnarray*}
Then Lemma 4 of \cite{JP} shows that 
\begin{eqnarray*}
\sum_{i=0}^j \E\left(W^n_iW^n_j\right) = \left\{
\begin{array}{ll} 
\kO\left(\sqrt{\frac{j}{n-j}}\right) & \textrm{if }d=3\\
\kO\left(\frac{\log j}{n-j}\right) & \textrm{if }d=4\\
\kO\left((n-j)^{1-d/2}\right) & \textrm{if }d\ge 5,
\end{array}
\right.
\end{eqnarray*}
and thus
\begin{eqnarray}
\label{WiWj}
\sum_{j=1}^{n-1}\sum_{i=0}^j \E\left(W^n_iW^n_j\right)= \left\{
\begin{array}{ll} 
\kO(n) & \textrm{if }d=3\\
\kO((\log n)^2) & \textrm{if }d=4\\
\kO(1) & \textrm{if }d\ge 5.
\end{array}
\right.
\end{eqnarray}
Now, for $i<j<n$, by using that $Z_i^j$ and $Z_j$ are independent, we get
$$\textrm{Cov}(Z_i,Z_j)=-\textrm{Cov}(W^j_i, Z_j).$$ 
On the other hand, assuming $i<j\le n$, 
\begin{eqnarray*}
\E(W^j_iZ_j)&=& \pp\left(j<H^{(i+1)}_{S_i+\overline V}<\infty,\, H^{(j+1)}_{S_j+\overline V}=\infty \right)\\
&=&\sum_{x\not\in \overline V}\, \pp(S_{j-i}=x,\, H_{\overline V}>j-i)\, 
\pp_x\big(H_{\overline V}<\infty,\, H_{x+\overline V}=\infty\big).
\end{eqnarray*}
Since in addition,
$$\E(Z_j)=\pp_x\left(H_{x+\overline V}=\infty\right)\quad \textrm{for all $x$},$$
and 
\begin{eqnarray*}
\E(W_i^j)= \sum_{x\notin \overline V} \, \pp \left( S_{j-i} =x,\, H_{\overline V}>j-i\right)\, \pp_x\left(H_{\overline V} <\infty\right),   
\end{eqnarray*}
we deduce that 
\begin{eqnarray*}
\textrm{Cov}(Z_i,Z_j)=\sum_{x\not\in \overline V}\pp(S_{j-i}=x,\, H_{\overline V}>j-i)\, b_V(x),
\end{eqnarray*}
with
\begin{eqnarray}
\label{JPb}
b_V(x):=\pp_x\left(H_{\overline V}<\infty\right)\, \pp_x\left(H_{x+\overline V}=\infty\right)-
\pp_x\left(H_{\overline V}<\infty,\ H_{x+\overline V}=\infty\right).
\end{eqnarray}
Now we need the following equivalent of Lemma 5 of 
\cite{JP}.
\begin{lem}
\label{lem-JPlike} 
For any $V\subset V_0$, and $x\notin \overline V$, 
\begin{eqnarray*}
b_V(x)=\pp_x\big(H_{\overline V}<H_{x+\overline V}<\infty\big)\pp_x(H_{\overline V}=\infty)+\kE(x,V),
\end{eqnarray*}
with
\begin{eqnarray*}
\kE(x,V):=\sum_{z\in x+\overline V}\pp_x\big(S_{H_{x+\overline V}}=z,\, H_{x+\overline V}<H_{\overline V}\big)
\Big(\pp_z\big(H_{\overline V}<\infty\big)-\pp_x\big(H_{\overline V}<\infty\big)\Big).
\end{eqnarray*}
Moreover, 
\begin{eqnarray*}
|\kE(x,V)|= \kO\left( \frac 1{\|x\|^{d-1}}\right).
\end{eqnarray*}
\end{lem}
\noindent Assuming this lemma for a moment, we get 
\begin{eqnarray*}
a_j &=&\sum_{i=0}^{j-1}\,  \textrm{Cov}(Z_i,Z_j)=\sum_{i=0}^{j-1}
\sum_{x\not\in \overline V}\,  \pp(S_{j-i}=x,\, H_{\overline V}>j-i)\,  b_V(x)\\
&=& \kO\left(\sum_{x\notin \overline V} \frac{G_{j}(0,x)}{\|x\|^{d-1}}\right) =   \kO\left(
\sum_{1\le \|x\|\le j} \frac 1{\|x\|^{2d-3}}\right) \\ 
&=& \left\{ 
\begin{array}{ll}
 \kO(\log j)& \textrm{if }d=3\\
\kO(1) & \textrm{if }d\ge 4,
\end{array}
\right.
\end{eqnarray*}
from which we deduce that 
\begin{eqnarray}
\label{ajs}
\sum_{j=0}^{n-1} a_j =  \left\{ 
\begin{array}{ll}
 \kO(n \log n)& \textrm{if }d=3\\
\kO(n) & \textrm{if }d\ge 4.
\end{array}
\right.
\end{eqnarray}
Then Proposition \ref{propRn} follows 
from \eqref{JP-var}, \eqref{WiWj} and \eqref{ajs}. 

\vspace{0.2cm}
\noindent \textit{Proof of Lemma \ref{lem-JPlike}.} 
Note first that 
\begin{eqnarray*}
b_V(x)&=& \pp_x\big(H_{\overline V}<\infty,\ H_{x+\overline V}<\infty\big)-
\pp_x\big(H_{\overline V}<\infty\big)\pp_x\big(H_{x+\overline V}<\infty\big)\\
&=& \pp_x\big(H_{\overline V}<H_{x+\overline V}<\infty\big)+
\pp_x\big(H_{x+\overline V}<H_{\overline V}<\infty\big)-
\pp_x\big(H_{\overline V}<\infty\big)\pp_x\big(H_{x+\overline V}<\infty\big).
\end{eqnarray*}
Moreover, 
\begin{eqnarray*}
\pp_x\big(H_{x+\overline V}<H_{\overline V}<\infty\big)
&=&\sum_{z\in x+\overline V}
\pp_x\big(S_{H_{x+\overline V}}=z,\ H_{x+\overline V}<H_{\overline V}\big)\pp_z\big(H_{\overline V}<\infty\big)\\
&=&  \pp_x\big(H_{x+\overline V}<H_{\overline V}\big)\pp_x(H_{\overline V}<\infty)+\kE(x,V).
\end{eqnarray*}
The first assertion of the lemma follows. The last assertion is then a direct consequence of standard asymptotics on the gradient of the Green's function (see for instance \cite[Corollary 4.3.3]{LL}).
\hfill $\square$

\begin{rem} \label{final.rem} \emph{By adapting the argument in \cite{JP} we could also prove that in dimension $3$, 
$\textrm{Var}(|\underline \kR_{n,V}|) \sim \sigma^2 n \log n$, for some constant $\sigma>0$, and then obtain a central limit theorem for this modified range. However it is not clear how to deduce from it an analogous result for $|\kR_{n,V}|$, which would be useful 
in view of a potential application to the boundary of the range.  }
\end{rem}

\vskip 1,cm
\noindent{\bf Acknowledgements.}
We would like to thank Gregory Maillard for discussions at an early
stage of this work.  We thank Pierre Mathieu for mentioning 
that we omitted to show that the limiting term 
in \eqref{prop-dyad-2} was nonzero, and Perla Sousi for mentioning 
a few other inaccuracies. 
Finally, we thank an anonymous referee for his very careful reading, 
and his numerous corrections and suggestions which 
greatly improved the presentation.
A.A. received support of the A$^*$MIDEX grant 
(ANR-11-IDEX-0001-02) funded by the French Government 
"Investissements d'Avenir" program.


\begin{thebibliography}{99}

\bibitem[AS]{AS} {\bf A. Asselah and B. Schapira}, 
{\it Moderate deviations for the range of a transient walk:
path concentration.} arXiv:1601.03957

\bibitem[BKYY]{BKYY} {\bf I. Benjamini, G. Kozma, A. Yadin, A. Yehudayoff}, {\it Entropy of random walk range},  Ann. Inst. Henri Poincar\'e Probab. Stat. 46 (2010), 1080--1092.

\bibitem[BY]{BY} {\bf N. Berestycki and A. Yadin}, {\it Condensation of random walks and the Wulff crystal},  arXiv:1305.0139. 


\bibitem[C]{C} {\bf X. Chen},
{\it Random Walk Intersections, Large Deviations and Related Topics.}
Mathematical Surveys and Monographs, Vol 157, 2009, AMS.

\bibitem[D]{Durrett} {\bf R. Durrett}, {\it Probability: theory and examples}, Cambridge University Press, Cambridge, Fourth edition, (2010). 

\bibitem[DE]{DE} {\bf  A. Dvoretzky, P. Erd\'os}, {\it Some problems on random walk in space}, Proceedings of the Second Berkeley Symposium on Mathematical Statistics and Probability, 1950, pp. 353--367. University of California Press, Berkeley and Los Angeles, (1951).


\bibitem[HSC]{HSC} {\bf W. Hebisch, L. Saloff-Coste}, {\it Gaussian estimates for Markov chains and random walks on groups},  Ann. Probab. 21, (1993), 673--709.


\bibitem[HA]{HA} {\bf J. M. Hammersley}, 
{\em Generalization of the Fundamental Theorem on Subadditive Functions}, 
Proc. Cambridge Phil. Society, 58, (1962), 235--238.

\bibitem[JP]{JP} {\bf N. C. Jain, W. E. Pruitt}, {\it The range of transient random walk}, J. Analyze Math. 24 (1971), 369--393.

\bibitem[LL]{LL} {\bf G. F. Lawler and V. Limic}, {\it Random walk: a modern introduction}. Cambridge Studies in Advanced Mathematics, 123. Cambridge University Press, Cambridge, (2010). xii+364 pp.

\bibitem[LG]{LG} {\bf J.-F. Le Gall}, {\it Propri\'et\'es d'intersection des marches al\'eatoires. I. Convergence vers le temps local d'intersection}, Comm. Math. Phys. 104 (1986), 471--507. 

\bibitem[MS]{MS} {\bf P. Mathieu, A. Sisto} {\it Deviation inequalities and CLT for random walks on acylindrically hyperbolic groups},  arXiv:1411.7865

\bibitem[Ok1]{Ok1} {\bf I. Okada} {\it The inner boundary of random walk range}, to appear in J. Math. Soc. Japan.  

\bibitem[Ok2]{Ok2} {\bf I. Okada} {\it Frequently visited site of the inner boundary of simple random walk range},  arXiv:1409.8368. 

\end{thebibliography}
\end{document}